\documentclass[a4paper,10pt]{article}
\usepackage{amsfonts}
\usepackage{bbm}
\usepackage{mathrsfs}
\usepackage{amsmath,amsthm,amssymb,amscd}
\usepackage[center]{titlesec}
\newtheorem{theo}{Theorem}[section]
\newtheorem{prop}[theo]{Proposition}
\newtheorem{lem}[theo]{Lemma}

\newtheorem{rem}[theo]{Remark}

\newtheorem{defi}[theo]{Definition}

\newtheorem{ques}[theo]{Question}
\newcommand{\bp}{\begin{proof}}
\newcommand{\ep}{\end{proof}}

\begin{document}
 \setlength{\baselineskip}{13pt} \pagestyle{myheadings}

 \title{
 {Symplectic Parabolicity and $L^2$ Symplectic Harmonic Forms}
 \thanks{
 Supported by NSFC (China) Grants 11471145, 11401514, 11701226 (Tan), 11371309, 11771377 (Wang);
 Natural Science Foundation of Jiangsu Province BK20170519 (Tan);
 University Scinence Research Project of Jiangsu Province 15KJB110024 (Zhou);
 Foundation of Yangzhou University 2015CXJ003 (Zhou).}
 }
 \author{{\large Qiang Tan, Hongyu Wang\thanks {E-mail:
 hywang@yzu.edu.cn}, Jiuru Zhou}\\
 ~
  {\small To the memory of Professor Weiyue Ding }\\
 }
 \date{}
 \maketitle

 \noindent {\bf Abstract:}
 In this paper, we study the symplectic cohomologies and symplectic harmonic forms which introduced by Tseng and Yau.
 Based on this, we get
 if $(M^{2n},\omega)$ is a closed symplectic parabolic manifold which satisfies the hard Lefschetz property, then its Euler number satisfies the inequality $(-1)^n\chi(M^{2n})\geq 0$.
 \\

 \noindent {{\bf AMS Classification (2000):} 53D05.}\\


 \section{Main results}
   This paper is related to a special case of the Chern conjecture claiming that the topological Euler characteristic of a real $2n$-dimensional
   closed manifold $M$ of negative curvature satisfies ${\rm sign}\chi(M)=(-1)^n$.
   This conjecture is true in dimensions $2$ and $4$ (\cite{Chern}).
   In dimension two, the answer follows immediately from the Gauss-Bonnet formula, i.e., a closed manifold of negative sectional curvature has negative Euler number.
   In dimension four, it is proved by J. Milnor (see \cite{Chern}) that negative sectional curvature implies that Gauss-Bonnet integrand is pointwise positive.

   A differential form $\alpha$ on a Riemannian manifold $(M,g)$ is called $d(bounded)$ if $\alpha$ is the exterior differential of a
   bounded form $\gamma$, i.e., $\alpha=d\gamma$, where $\|\gamma\|_{L^{\infty}}=\sup_{x\in M}\| \gamma(x)\|_g<\infty$.
    A form $\alpha$ on a Riemannian manifold $(M,g)$ is called $\tilde{d}(bounded)$ if the lift $\tilde{\alpha}$ of $\alpha$ to the
    universal covering $\tilde{M}\rightarrow M$ is $d(bounded)$ on $\tilde{M}$ with respect to the lift metric $\tilde{g}$.
   Gromov gave the definition of K\"{a}hler hyperbolic in \cite{Gromov}.
   A closed complex manifold is called K\"{a}hler hyperbolic if it admits a K\"{a}hler metric whose K\"{a}hler form $\omega$ is $\tilde{d}(bounded)$.
   Similarly, we can define symplectic hyperbolic manifold.
    Let $(M,\omega)$ be a closed symplectic manifold.
    Choose a $\omega$-compatible almost complex structure $J$ on $M$ (\cite{MS}).
     Define an almost K\"{a}hler metric, $g(\cdot,\cdot)=\omega(\cdot,J\cdot)$, on $M$. Then the triple $(g,J,\omega)$ is called an almost K\"ahler structure on $M$ and the quadraple $(M,g,J,\omega)$ is called a closed almost K\"ahler manifold.

   \begin{defi}
   A closed almost K\"ahler manifold $(M,g,J,\omega)$ is called symplectic hyperbolic if the lift $\tilde{\omega}$ of $\omega$
    to the universal covering
   $(\tilde{M},\tilde{g},\tilde{J},\tilde{\omega})\rightarrow (M,g,J,\omega)$
   is $d(bounded)$ on $(\tilde{M},\tilde{g},\tilde{J},\tilde{\omega})$.
   \end{defi}

   M. Gromov \cite{Gromov} introduced the notion of K\"{a}hler hyperbolicity and proved the above conjecture in the K\"{a}hler case.
   After Gromov's work, J. Jost and K. Zuo \cite{Jost} obtained that:
   If $M$ is a $2n$-dimensional closed Riemannian manifold of non-positive sectional curvature and homotopy equivalent to a closed
   K\"{a}hler manifold, then $(-1)^n\chi(M)\geq0$.
   It is well known that the geodesic flow on the unit tangent bundle of a negatively curved
   closed Riemannian manifold is an Anosov geodesic flow.
   Cheng \cite{Cheng} has proven that: Let $M$ be a $2n$-dimensional closed Riemannian manifold with Anosov geodesic flow.
   If $M$ is homotopy equivalent to a closed K\"{a}hler manifold, then $(-1)^n\chi(M)>0$.

    Gromov has proven that:
   If $(M,g)$ is complete simply connected and has strictly negative sectional curvature,
   then every smooth bounded closed from of degree $k\geq 2$ is $d(bounded)$ (\cite[0.1 B]{Gromov}).
   We want to single out a condition which is weaker than $d$-boundedness
   and can be applied to K\"{a}hler manifolds of non-positive curvature.
   A differential form $\alpha$ on a closed Riemannian manifold is called $d(sublinear)$ (cf. \cite{Hitchin,Jost})
    if there exist a differential form $\beta$ and a number $c>0$
   such that $\alpha=d\beta$ and $\|\beta(x)\|_g\leq c(\rho(x_0,x)+1)$, where $\rho(x_0,x)$ stands
   for the Riemannian distance between $x$ and a base point $x_0$.
   A form $\alpha$ on a Riemannian manifold $(M,g)$ is called $\tilde{d}(sublinear)$ if the lift $\tilde{\alpha}$ of $\alpha$ to the
    universal covering $\tilde{M}\rightarrow M$ is $d(sublinear)$ on $\tilde{M}$ with respect to the lift metric $\tilde{g}$.
   We pay special attention to the symplectic manifold.

   \begin{defi}
    A closed almost K\"ahler manifold $(M,g,J,\omega)$ is called symplectic parabolic
    if the lift $\tilde{\omega}$ of $\omega$
    to the universal covering
   $(\tilde{M},\tilde{g},\tilde{J},\tilde{\omega})\rightarrow (M,g,J,\omega)$
   is $d(sublinear)$  on $(\tilde{M},\tilde{g},\tilde{J},\tilde{\omega})$.
   \end{defi}

   \vskip 6pt

   If $L^2\Omega^k$ denotes the Hilbert space of $L^2$ $k$-forms,
   then the $L^2$-cohomology group $L^2H^k_{dR}$ is defined as the quotient of the space of closed $L^2$ $k$-forms by the
   closure of the space $dL^2\Omega^{k-1}\cap L^2\Omega^k$.
   It is a theorem that on a complete manifold any harmonic $L^2$ $k$-form is closed and coclosed
    and so represents a class in $L^2H^k_{dR}$.
    Hitchin \cite{Hitchin} has proven that the $L^2$ harmonic forms on a complete noncompact K\"{a}hler parabolic manifold lie in the
   middle dimension, that is,
   if the K\"{a}hler form $\omega$ on a complete noncompact K\"{a}hler manifold
   is $d(sublinear)$, then the only $L^2$ harmonic forms lie in the middle dimension.
   In this paper, we want to prove some similar results for another two $L^2$ harmonic forms ($L^2$ symplectic harmonic).

   \vskip 12pt

  Let $(M,g,J,\omega)$ be a closed almost K\"ahler $2n$-manifold, that is,
  $M$ is a closed differential manifold with an almost K\"ahler structure on $M$.
  Symplectic Hodge theory was introduced by Ehresmann and Libermann \cite{EL} and was rediscovered by Brylinski \cite{BrA}.
  They defined the symplectic star operator $*_s: \Omega^k(M)\rightarrow \Omega^{2n-k}(M)$ analogously to the Hodge
  star operator but with respect to the symplectic form $\omega$.
  As in Riemannian Hodge theory, define $d^{\Lambda}=(-1)^{k+1}*_s d*_s$ on $\Omega^k(M)$ (cf. \cite{Koszul}).
  A form $\alpha$ is called symplectic harmonic if it satisfies $d\alpha=d^{\Lambda}\alpha=0$.
  Brylinski conjectured that on a closed symplectic manifold, every de Rham cohomology class contains a symplectic harmonic representative.
  Some evidence for his conjecture was presented in his paper \cite{BrA} and he proved the conjecture for closed K\"{a}hler manifolds. Several years later,
  his conjecture for closed symplectic manifolds was disproved by Oliver Mathieu \cite{Mathieu}.
  Brylinski's conjecture is equivalent to the question of the existence of a Hodge decomposition in the symplectic sense.
  The uniqueness of the decompostion in this case is evidently not true.
  Mathieu gave two ways to give counter-examples to Brylinski's conjecture.
  In fact, Mathieu proved that every de Rham cohomology $H^*_{dR}(M)$ class contains a symplectic harmonic form if and only if the
  symplectic manifold satisfies the hard Lefschetz property, that is, the map
   \[
    H^k_{dR}(M)\rightarrow H^{2n-k}_{dR}(M),\quad A\mapsto [\omega]^{n-k}\wedge A,
   \]
   is an isomorphism for all $k\leqslant n$.

   Mathieu's theorem is a generalization of the hard Lefschetz theorem for closed K\"{a}hler manifolds.
  His proof involves the representation theory of quivers and Lie superalgebras.
  Dong Yan \cite{YanH} provided a simpler, more direct proof of this fact.
  Yan's proof follows the idea of the standard proof of the hard Lefschetz theorem.

  Both the existence and uniqueness of symplectic harmonic forms may be not expected to hold in de Rham cohomology for closed almost K\"ahler manifolds.
  So Tseng and Yau thought that the de Rham cohomology may be not the appropriate cohomology to consider symplectic Hodge theory.
   They \cite{TY} introduced some new cohomology groups $H^k_{d+d^\Lambda}(M)$ and $H^k_{dd^\Lambda}(M)$
   for a symplectic manifold $(M,\omega)$.
   These two cohomologies are similarly paired and share many analogous properties with the pair Bott-Chern cohomology
   and Aeppli cohomology defined on complex manifolds.
   Indeed, both can be shown to be finite dimensional on closed complex manifolds
   by constructing self-adjoint fourth-order differential operators (cf. \cite{Kodaira}).
   Similarly to the construction in \cite{Kodaira}, Tseng and Yau found out the associated Laplacian operators $\Delta_{d+d^\Lambda}$ and $\Delta_{dd^\Lambda}$ with respect to an almost K\"ahler structure $(g,J,\omega)$ of these new cohomologies.
   Unfortunately, $\Delta_{d+d^\Lambda}$ and $\Delta_{dd^\Lambda}$ are not elliptic operators.
   Then they introduce elliptic operators $D_{d+d^\Lambda}$ and $D_{dd^\Lambda}$ such that
   $\ker\Delta_{d+d^\Lambda}=\ker D_{d+d^\Lambda}$, $\ker\Delta_{dd^\Lambda}=\ker D_{dd^\Lambda}$.
   If $(M,\omega)$ is closed, Tseng and Yau have proven that both $H^k_{d+d^\Lambda}(M)$ and $H^k_{dd^\Lambda}(M)$ are finite dimensional
  (see \cite[Theorem 3.5, 3.16]{TY}).
   Then we can define the $kth$ symplectic Betti numbers
   $$\beta^{s,1}_k\triangleq\dim H^k_{d+d^\Lambda}(M),\,\,\,\beta^{s,2}_k\triangleq\dim H^k_{dd^\Lambda}(M)$$
    and the symplectic Euler numbers
    $$\chi^{s,1}(M)\triangleq\sum^{2n}_{k=0}(-1)^k\beta^{s,1}_k,\,\,\,\chi^{s,2}(M)\triangleq\sum^{2n}_{k=0}(-1)^k\beta^{s,2}_k.$$

  For an almost K\"ahler manifold, we denote the spaces of $d+d^\Lambda$ harmonic $k$-forms and $dd^\Lambda$ harmonic $k$-forms by
  $\mathcal{H}^k_{d+d^\Lambda}(M)$ and $\mathcal{H}^k_{dd^\Lambda}(M)$
  that are the kernel spaces of $\Delta_{d+d^\Lambda}$ and $\Delta_{dd^\Lambda}$, respectively.
   If $(M,g,J,\omega)$ is closed, Tseng and Yau gave the Hodge decompositions for $\mathcal{H}^k_{d+d^\Lambda}(M)$ and $\mathcal{H}^k_{dd^\Lambda}(M)$.
   Then they got $H^k_{d+d^\Lambda}(M)\cong\mathcal{H}^k_{d+d^\Lambda}$
   and $H^k_{dd^\Lambda}(M)\cong\mathcal{H}^k_{dd^\Lambda}$ on a closed symplectic manifold.
   Here, our first main result is considered on the complete noncompact almost K\"ahler manifold $(M,g,J,\omega)$.
   In the following section, the notation $L^2$ on $(M,g,J,\omega)$
   is meant with respect to the almost K\"ahler metric $g(\cdot,\cdot)=\omega(\cdot,J\cdot)$, where $J$ is an almost complex structure on $M$ compatible with $\omega$.

   Suppose that $(M,g,J,\omega)$ is a complete non-compact almost K\"ahler manifold of $2n$-dimension.
   In general, $J$ is not integrable, hence $\nabla J\neq 0$, $\nabla \omega\neq 0$, where $\nabla$ is the Levi-Civita connection induced
   from the metric $g$ (cf. \cite{Cha}).
   An almost K\"ahler manifold $(M,g,J,\omega)$ is of bounded geometry if $(\nabla)^kJ$,  $(\nabla)^k\omega$ have uniformly point-wise bounded on $M$.
   There are many complete non-compact almost K\"ahler manifolds with bounded geometry, for example, the universal covering
   $(\tilde{M},\tilde{g},\tilde{J},\tilde{\omega})$ of a closed almost K\"ahler manifold $(M,g,J,\omega)$ whose $\pi_1(M)$ is infinite
   is a noncompact almost K\"ahler manifold with bounded geometry.

  Notice that if $(M,g,J,\omega)$ is a noncompact almost K\"ahler manifold with bounded geometry,
  Hilbert spaces $L^2_l\Omega^k$ are the closure of $\Omega^k_c(M)$, $l\in\mathbb{Z}^+$, where $\Omega^k_c(M)$ is the space of
  $C^\infty$ $k$-forms with compact support in $M$ (cf. \cite[Remark 2.7]{Aubin}).

   Let $H$ be a Hilbert space and $T: dom(T)\rightarrow H$ be a (not necessarily bounded)
   linear operator defined on a dense linear subspace $dom(T)$ which is called (initial) domain.
   We call $T$ closed if its graph $gr(T)\triangleq \{(u,T(u)):u\in dom(T)\}\subset H\times H$ is closed.
   We say that $S: dom(S)\rightarrow H$ is an extension of $T$ and write $T\subset S$ if $dom(T)\subset dom(S)$
   and $S(u)=T(u)$ holds for all $u\in dom(T)$.
   We write $T=S$ if $dom(T)=dom(S)$ and $S(u)=T(u)$ holds for all $u\in dom(T)$.
   We call $T$ closable if and only if $T$ has a closed extension.
   Since the intersection of an arbitrary family of closed sets is closed again, a closable unbounded densely defined operator $T$
   has a unique minimal closure, also called minimal closed extension, i.e., a closed operator $T_{min}:dom(T_{min})\rightarrow H$
   which $T\subset T_{min}$ such that $T_{min}\subset S$ holds for any closed extension $S$ of $T$.
   Explicitly $dom(T_{min})$ consists of elements $u\in H$ for which these exists a sequence $(u_n)_{n\geq 0}$ in $dom(T)$ and an element
   $v$ in $H$ satisfying $\lim_{n\rightarrow\infty}u_n=u$ and $\lim_{n\rightarrow\infty}T(u_n)=v$.
   Then $v$ is uniquely determined by this property and we put $T_{min}(u)=v$.
   Equivalently, $dom(T_{min})$ is the Hilbert space completion of $dom(T)$ with respect to the inner product
   \begin{equation}
     <u,v>_{gr}= <u,v>_H+ <T(u),T(v)>_H.
   \end{equation}
   If not stated otherwise we always use the minimal closed extension as the closed extension of a closable unbounded densely defined
   linear operator.

   The adjoint of $T$ is the operator $T^*: dom(T^*)\rightarrow H$ whose domain consists of elements $v\in H$ for which there is an element
   $u$ in $H$ such that $<u',u>_H=<T(u'),v>_H$ holds for all $u'\in dom(T)$.
   Then $u$ is uniquely determined by this property and we put $T^*(v)=u$.
   Notice that $T^*$ may not have a dense domain in general.
   If $T$ is closable, then $T^*_{min}=T^*$ and $T_{min}=(T^*)^*$.
   We call $T$ symmetric if $T\subset T^*$ and self-adjoint if $T=T^*$.
   Any self-adjoint operator is necessarily closed and symmetric.
   A bounded operator $T:H\rightarrow H$ is always closed and is self-adjoint if and only if it is symmetric.
   We call $T$ essentially self-adjoint if $T_{min}$ is self-adjoint.
   The maximal closure $T_{max}$ of $T$ is defined by the adjoint of $(T^*)_{min}$.
  In fact, $dom(T_{max})$ is the space of all $u\in H$ such that $Tu\in H$ (as a distrbution).
   For any closure $\bar{T}$ of $T: dom(T)\rightarrow H$ we have $T_{min}\subset\bar{T}\subset T_{max}$.
   Hence if $T$ is essentially self-adjoint, then $T_{min}=T_{max}$.
   More details, see \cite{Ati,L¨¹ck,Shubin}.

   \begin{theo}\label{main 1}
   If $(M,g,J,\omega)$ is a $2n-$dimensional complete noncompact almost K\"ahler manifold with bounded geometry
   whose symplectic form $\omega$ is $d(sublinear)$,
   and $D_{d+d^\Lambda}$, $D_{dd^\Lambda}$ are essentially self-adjoint elliptic operators on $M$,
   then their $L^2$ symplectic harmonic forms satisfy $L^2\mathcal{H}^k_{d+d^\Lambda}=0$ and $L^2\mathcal{H}^k_{dd^\Lambda}=0$ , unless $k=n$.
   \end{theo}
   In general, the symplectic Betti numbers $\beta^s_k$, $0\leq k\leq 2n$, are not topological invariants.
   If the closed symplectic manifold $(M,\omega)$ satisfies the hard Lefschetz property, then
    $$H^k_{d+d^\Lambda}(M)\cong H^k_{dR}(M)\cong H^{2n-k}_{dd^\Lambda}(M)$$
   and $\chi(M)=\chi^{s,1}(M)=\chi^{s,2}(M)$.
   Suppose that $(M,\omega)$ is a closed symplectic manifold and $(\tilde{M},\tilde{\omega})$ is the universal covering space.
   We consider $L^2$ symplectic harmonic forms on $(\tilde{M},\tilde{\omega})$ and define $L^2$-symplectic Euler characteristics
  $L^2\chi^{s,1}(\tilde{M)}$,
   $L^2\chi^{s,2}(\tilde{M})$ on it.
    Then we get the following theorem.
   \begin{theo}\label{main 3}
    Let $(M,\omega)$ be a $2n$-dimensional closed symplectic manifold,
    $\Pi:(\tilde{M},\tilde{\omega})\rightarrow (M,\omega)$ the universal covering map.
    If $(M,\omega)$ satisfies the hard Lefschetz property, then
    $$
    L^2\chi^{s,1}(\tilde{M)}=L^2\chi^{s,2}(\tilde{M})=L^2\chi(\tilde{M})=\chi(M)=\chi^{s,1}(M)=\chi^{s,2}(M).
    $$
   \end{theo}

  At last, we want to consider Chern conjecture on a closed symplectic parabolic manifold.
  One of the powerful tools for Gromov achieving Chern conjecture on a K\"{a}hler manifold is that the Lefschetz operator commutes with the Hodge Laplacian operator $\Delta_d=dd^*+d^*d$.
   But in general, the Lefschetz operator does not commute with $\Delta_d$ on symplectic manifold.
   This makes us think that de Rham cohomology and its harmonic forms are inappropriate to be seen as a tool to solve problems on symplectic manifolds.
   Fortunately, by considering Tseng and Yau's new symplectic cohomologies on symplectic parabolic manifold, we get some interesting results.
   At last, with the hard Lefschetz property which ensures that de Rham cohomology consists with the new symplectic cohomology,
   we can obtain the third main result.
  \begin{theo}\label{main 2}
   If $(M,\omega)$ is a $2n$-dimensional closed symplectic parabolic manifold which satisfies the hard Lefschetz property,
    then the Euler number satisfies $(-1)^n\chi(M)\geq0$.
   \end{theo}

  \begin{rem}
  1) Jianguo Cao and Frederico Xavier in {\rm \cite{Cao}}(see also J. Jost and K. Zuo's work {\rm \cite{Jost}}) have proven that a bounded closed $k$-form, $k\geq1$, is $d(sublinear)$ on
  a complete simply-connected manifold of non-positive sectional curvature.
  By their Lemma 3 in  {\rm \cite{Cao}}, we can also get:
  Let $M$ be a closed $2n$-Riemannian manifold of non-positive sectional curvature.
  If $M$ is homotopy equivalent with a closed symplectic manifold which satisfies the hard Lefschetz property,
   then the Euler number of $M$ satisfies the inequality $(-1)^n\chi(M)\geq0$.

   2)It is well known that a closed K\"{a}hler manifold $M$ such that $\pi_1(M)$ is
   word-hyperbolic in the sense of {\rm\cite{Gromov1}} and $\pi_2(M)=0$ is a  K\"{a}hler hyperbolic manifold.
   Hence, we conjecture that:
    Let $(M,\omega)$ be a $2n$-dimensional closed symplectic manifold,
    if $\pi_1(M)$ is infinite and $\pi_2(M)=0$, then $\omega$ is $\tilde{d}(bounded)$.

  \end{rem}
 Since the hard Lefschetz property in this article is a technical condition, we have the following question:
 \begin{ques}
 If we drop the condition that $(M,\omega)$ satisfies the hard Lefschetz property in {\rm Theorem \ref{main 3}} and {\rm Theorem \ref{main 2}},
 could the same conclusion hold?
 \end{ques}

 \section{$L^2$ symplectic Hodge theory}\setcounter{equation}{0}
   Let us recall some definitions and some results of Hodge theory.
  Let $M$ be a closed oriented Riemannian manifold with metric $g$.
  The Hodge star operator $*_g: \Omega^k(M)\rightarrow \Omega^{m-k}(M)$ is a linear map which satisfies $\alpha\wedge *_g\beta=(\alpha,\beta)_{g} d\text{vol}_g$ for all $\alpha,\beta\in\Omega^k(M)$.
  Here $\Omega^k(M)$ is the space of the smooth $k$-forms on $M$.
  We denote the adjoint operator of the differential operator $d$ by $d^\ast$ associated to $g$.
  By a direct calculation, we will find that $d^\ast=(-1)^{mk+m+1}*_gd*_g$ on $\Omega^k(M)$.
  A form $\alpha$ is called harmonic if it is both $d$-closed and $d^\ast$-closed.
  The Laplacian operator is given by $\Delta_g=dd^*+d^*d:\Omega^k(M)\rightarrow\Omega^k(M)$,
  then $\alpha$ is harmonic if and only if $\Delta_g\alpha=0$. By the theory of elliptic operator we conclude that the kernel of $\Delta_g$ is finite dimensional.
  And the Hodge decomposition tells us every cohomology class has a unique harmonic representative.

   Let $(M,\omega)$ be a closed symplectic $2n$-manifold.
   Symplectic Hodge theory was introduced by Ehresmann and Libermann \cite{EL} and was rediscovered by Brylinski \cite{BrA}.
  They defined the symplectic star operator analogously to the Hodge star operator but with respect to the symplectic form $\omega$.
  The symplectic star operator $*_s$ acts on a differential
  $k$-form $\alpha$ by
  \begin{eqnarray*}
  \alpha\wedge*_s\alpha' &=&(\omega^{-1})^k(\alpha,\alpha')d{\rm vol} \\
  &=& \frac{1}{k!}(\omega^{-1})^{i_1j_1}\cdot\cdot\cdot(\omega^{-1})^{i_kj_k}\alpha_{i_1\cdot\cdot\cdot i_k}\alpha'_{j_1\cdot\cdot\cdot j_k}\frac{\omega^n}{n!}
  \end{eqnarray*}
  with repeated indices summed over.
  Direct calculation shows
  \begin{equation}\label{formula}
   \alpha\wedge*_s\beta=(-1)^k\beta\wedge*_s\alpha,
  \end{equation}
   where $\alpha$ and $\beta$ are $k$-forms.
  The adjoint of the standard exterior derivative with respect to $\omega$ takes the form (cf. \cite{Koszul})
  $$
  d^\Lambda=(-1)^{k+1}*_sd*_s.
  $$
  Fix an almost K\"ahler structure $(g,J,\omega)$ on $(M,\omega)$.
  See some standard Hodge adjoint of the differential operators.
  Denote by
  $$d^*=-*_gd*_g,$$ $$d^{\Lambda*}=*_gd^{\Lambda}*_g,$$ and $$(dd^{\Lambda})^*=(-1)^{k+1}*_gdd^{\Lambda}*_g$$
  act on $k$-forms.
  By using the properties $d^2=(d^\Lambda)^2=0$ and the anti-commutively $dd^\Lambda=-d^\Lambda d$,
  we will find that any form that is $dd^\Lambda$-exact is also $d$- and $d^\Lambda$-closed.
  This gives a differential complex
  \begin{equation}\label{complex}
  \Omega^k\xrightarrow{dd^\Lambda}\Omega^k\xrightarrow{d+d^\Lambda}\Omega^{k+1}\oplus\Omega^{k-1}.
  \end{equation}

  Tseng and Yau \cite{TY} considered the symplectic cohomology group $H^k_{d+ d^\Lambda}(M)$ which are just the symplectic version of well-known cohomologies in complex geometry already studied by Kodaira and Spencer \cite{Kodaira}, for example. With complex (\ref{complex}), they define
  \begin{equation}\label{symplectic coho}
    H^k_{d+d^{\Lambda}}(M)=\frac{\ker(d+d^{\Lambda})\cap\Omega^k(M)}{\text{im}\,dd^{\Lambda}\cap\Omega^k(M)}.
  \end{equation}
  From the differential complex, the Laplacian operator associated with
  the cohomology is
  $$
  \Delta_{d+d^{\Lambda}}=dd^{\Lambda}(dd^{\Lambda})^*+\lambda(d^*d+d^{\Lambda*}d^{\Lambda}),
  $$
  where we have inserted an undetermined real constant $\lambda> 0$ that gives the relative weight between the terms.
  The Laplacian is a fourth-order self-adjoint differential operator but not elliptic.
  However, Tseng and Yau introduce a related fourth-order elliptic operator (cf. \cite{Kodaira,TY})
  \begin{equation}\label{elliptic}
    D_{d+d^{\Lambda}}=dd^{\Lambda}(dd^{\Lambda})^*+(dd^{\Lambda})^*(dd^{\Lambda})+d^*d^{\Lambda}d^{\Lambda*}d
    +d^{\Lambda*}dd^*d^{\Lambda}+\lambda(d^*d+d^{\Lambda*}d^{\Lambda}).
  \end{equation}
  The solution space of $D_{d+d^{\Lambda}}\alpha=0$ is identical to that of $\Delta_{d+d^{\Lambda}}\alpha=0$.
  A differential form $\alpha$ is called $d+d^\Lambda$-harmonic if $\Delta_{d+d^{\Lambda}}\alpha=0$,
  or equivalently,
  $$
  d\alpha=d^{\Lambda}\alpha=0 \,\,\, {\rm and} \,\,\, (dd^{\Lambda})^*\alpha=0.
  $$
  Denote the space of $d+d^\Lambda$-harmonic $k$-forms by $\mathcal{H}^k_{d+d^\Lambda}(M)$.
  Tseng and Yau \cite{TY} proved that the space of $d+d^\Lambda$-harmonic $k$-forms $\mathcal{H}^k_{d+d^\Lambda}(M)$ are finite dimensional and
  isomorphic to $H^k_{d+d^\Lambda}(M)$.
  \begin{prop}\label{Yau theorem}
  ({\rm\cite[Theorem 3.5]{TY}})
  Let $(M,g,J,\omega)$ be a closed almost K\"ahler manifold.
  Then:

  (1) $\dim\mathcal{H}^k_{d+d^\Lambda}(M)<\infty$.

  (2) There is an orthogonal decomposition
  $$
  \Omega^k=\mathcal{H}^k_{d+d^\Lambda}\oplus dd^\Lambda\Omega^k\oplus(d^*\Omega^{k+1}+d^{\Lambda*}\Omega^{k-1}).
  $$

  (3) There is a caninical isomorphism: $\mathcal{H}^k_{d+d^\Lambda}(M)\cong H^k_{d+d^\Lambda}(M)$.
  \end{prop}

   Interestingly, simply reversing the arrows of the complex (\ref{complex}) leads to
   another symplectic cohomology group $H^k_{dd^\Lambda}(M)$ (cf. \cite{TY}),
  \begin{equation}\label{symplectic coho2}
    H^k_{dd^{\Lambda}}(M)=\frac{\ker(dd^{\Lambda})\cap\Omega^k(M)}{({\rm im}d+{\rm im}d^{\Lambda})\cap\Omega^k(M)}.
  \end{equation}
  The Laplacian operator associated with
  the cohomology is
   $$
  \Delta_{dd^{\Lambda}}=(dd^{\Lambda})^*dd^{\Lambda}+\lambda(dd^*+d^{\Lambda}d^{\Lambda*}).
  $$
  The Laplacian is also not elliptic.
  A differential form $\alpha$ is called $dd^\Lambda$-harmonic if $\Delta_{dd^{\Lambda}}\alpha=0$,
  or equivalently,
  $$
  d^*\alpha=0,\,\,\,d^{\Lambda*}\alpha=0 \,\,\, {\rm and} \,\,\, dd^{\Lambda}\alpha=0.
  $$
  Denote the space of $dd^\Lambda$-harmonic $k$-forms by $\mathcal{H}^k_{dd^\Lambda}(M)$.
   Then Tseng and Yau introduce a fourth-order elliptic operator
  \begin{equation}\label{elliptic 2}
    D_{dd^{\Lambda}}=(dd^{\Lambda})^*dd^{\Lambda}+(dd^{\Lambda})(dd^{\Lambda})^*+dd^{\Lambda*}d^{\Lambda}d^*
    +d^{\Lambda}d^*dd^{\Lambda*}+\lambda(dd^*+d^{\Lambda}d^{\Lambda*}),
  \end{equation}
  which satisfies $\ker D_{dd^{\Lambda}}=\ker\Delta_{dd^{\Lambda}}=\mathcal{H}^k_{dd^\Lambda}(M)$.
  Tseng and Yau \cite{TY} proved that the space of $dd^\Lambda$-harmonic $k$-forms $\mathcal{H}^k_{dd^\Lambda}(M)$ are finite dimensional and
  isomorphic to $H^k_{dd^\Lambda}(M)$.
  \begin{prop}\label{Yau theorem2}
  ({\rm\cite[Theorem 3.16]{TY}})
  Let $(M,g,J,\omega)$ be a closed almost K\"ahler manifold.
  Then:

  (1) $\dim\mathcal{H}^k_{dd^\Lambda}(M)<\infty$.

  (2) There is an orthogonal decomposition
  $$
  \Omega^k=\mathcal{H}^k_{dd^\Lambda}\oplus (dd^\Lambda)^*\Omega^k\oplus(d\Omega^{k-1}+d^{\Lambda}\Omega^{k+1}).
  $$

  (3) There is a caninical isomorphism: $\mathcal{H}^k_{dd^\Lambda}(M)\cong H^k_{dd^\Lambda}(M)$.
  \end{prop}

  Using the symplectic form $\omega=\sum\frac{1}{2}\omega_{ij}dx^i\wedge dx^j$,
  the Lefschetz operator $L: \Omega^k(M)\rightarrow \Omega^{k+2}(M)$ and the dual Lefschetz operator
  $\Lambda:\Omega^k(M)\rightarrow\Omega^{k-2}(M)$ are defined by
 \begin{equation*}
 \begin{aligned}
 &L:\ L(\alpha)=\omega\wedge\alpha,\\
 &\Lambda:\ \Lambda(\alpha)=\frac{1}{2}(\omega^{-1})^{ij}i_{\partial_{x^i}}i_{\partial_{x^j}}\alpha,
 \end{aligned}
 \end{equation*}
 where $i$ denotes the interior product.

    \begin{prop}\label{isomorphic}
   ({\rm\cite[Corollary 3.8, 3.19]{TY}})
   On a closed almost K\"ahler manifold $(M,g,J,\omega)$ of dimension $2n$,
   the Lefschetz operator defines isomorphisms
   $$
   L^{n-k}:\mathcal{H}^k_{d+d^\Lambda}\cong \mathcal{H}^{2n-k}_{d+d^\Lambda}
   $$
   and
    $$
   L^{n-k}:\mathcal{H}^k_{dd^\Lambda}\cong \mathcal{H}^{2n-k}_{dd^\Lambda}
   $$
   for $k\leq n$.
    \end{prop}

   \vskip 6pt

  The compactness becomes important when one integrates by parts.
  For example, by applying the Stokes formula
  \begin{equation}\label{Stokes}
    \int_Md(\varphi\wedge*_g\psi)=0,
  \end{equation}
  we can derive the desired relation
  $<d\varphi,\psi>_g=<\varphi,d^*\psi>_g$.
  If $M$ is noncompact, then (\ref{Stokes}) is not true generally.
  Fortunately, Gromov has proven that (\ref{Stokes}) remains true for all $L^1$-forms on a complete manifold.
  \begin{lem}\label{Gromov lemma}
  ({\rm\cite[Lemma 1.1.A]{Gromov}})
  Suppose $M$ is a complete $m$-manifold.
  Let $\alpha$ be an $L^1$-form on $M$ of degree $m-1$ such that the differential $d\alpha$ is also $L^1$.
  Then
  $$
  \int_Md\alpha=0.
  $$
   \end{lem}

   \begin{rem}
   The above relation for $C^\infty$ forms easily yields the statement for nonsmooth $\eta$ where $d\eta$ is understood as a distribution
    (cf. {\rm \cite{Gromov,L¨¹ck}}).
   \end{rem}

  Let $\Delta_d=dd^*+d^*d$ be the de Rham Laplacian.
  Denote the space of $\Delta_d$-harmonic $k$-forms by $\mathcal{H}^k_d(M)$.
  In \cite{Gromov}, Gromov has gotten
  \begin{equation}\label{Hodge deco}
     L^2\Omega^k=L^2\mathcal{H}^k_d\oplus
     [\overline{d(L^2\Omega^{k-1})}\oplus\overline{d^*(L^2\Omega^{k+1})}].
  \end{equation}
  Along the lines used by Gromov(see also \cite{L¨¹ck}), we want to obtain another two decompositions.
   \begin{lem}\label{our lemma}
   Suppose that $(M,g,J,\omega)$ is a complete noncompact almost K\"ahler manifold with bounded geometry.
   Then we can get:\\
   (1) $<d\alpha,\beta>_g=<\alpha,d^*\beta>_g$, for $\alpha,\beta,d\alpha,d^*\beta\in L^2\Omega^*(M)$,\\
   (2) $<d^\Lambda\alpha,\beta>_g=<\alpha,d^{\Lambda*}\beta>_g$, for $\alpha,\beta,d^\Lambda\alpha,d^{\Lambda*}\beta\in L^2\Omega^*(M)$,\\
   (3) $<dd^\Lambda\alpha,\beta>_g=<\alpha,(dd^{\Lambda})^*\beta>_g$,
   for $\alpha,\beta,d^\Lambda\alpha,d^*\beta,dd^\Lambda\alpha,(dd^{\Lambda})^*\beta\in L^2\Omega^*(M)$.
   \end{lem}
   \begin{proof}
   (1) By observing the formula
    $$d\alpha\wedge*\beta-\alpha\wedge*d^*\beta=\pm d(\alpha\wedge*\beta),$$
   and note that both $\alpha\wedge*\beta$ and $d(\alpha\wedge*\beta)$ are $L^1$-forms,
   we can easily get
    $$<d\alpha,\beta>_g=<\alpha,d^*\beta>_g$$
   by applying the Lemma \ref{Gromov lemma}.

   (2) Suppose that $\alpha$ is a $k$-form and $\beta$ is a  $k-1$-form.
    \begin{eqnarray*}
    (d^\Lambda\alpha,\beta)_gdvol_g&=& (-1)^{k+1}*_sd*_s\alpha\wedge*_g\beta \\
     &=& (-1)^{(k-1)^2}d*_s\alpha\wedge*_s*_g\beta \\
     &=& (-1)^{(k-1)^2}[d(*_s\alpha\wedge*_s*_g\beta)+(-1)^{k-1}*_s\alpha\wedge d*_s*_g\beta]
   \end{eqnarray*}
   By the assumption of conditions,
   we find that both $*_s\alpha\wedge*_s*_g\beta$ and $d(*_s\alpha\wedge*_s*_g\beta)$ are $L^1$-forms.
   Taking integral of both sides of the above equation,
   we obtain
   \begin{eqnarray*}
    <d^\Lambda\alpha,\beta>_g &=& \int_M(-1)^{(k-1)^2}d(*_s\alpha\wedge*_s*_g\beta)+\int_M(-1)^{k^2-k}*_s\alpha\wedge d*_s*_g\beta \\
    &=& (-1)^{k^2-k}\int_M*_s\alpha\wedge d*_s*_g\beta\\
     &=&\int_M\alpha\wedge *_sd*_s*_g\beta\\
     &=&\int_M\alpha\wedge (-1)^kd^{\Lambda}*_g\beta\\
     &=&<\alpha,(-1)^{k^2+k}*_gd^{\Lambda}*_g\beta>_g\\
     &=&<\alpha,d^{\Lambda*}\beta>_g.
     \end{eqnarray*}

     (3) The third conclusion is an obvious result following (1) and (2).
   \end{proof}
   With Lemma \ref{our lemma}, we can obtain another very useful lemma.
   Before giving the useful lemma, we claim that there exists a family of cutoff functions $a_\varepsilon$
   such that
   $$
   a_\varepsilon\geq0,\,\,\,\|\nabla^1_g a_\varepsilon\|_g\leq\varepsilon (a_\varepsilon)^{\frac{1}{2}},
   \,\,\, \|(\nabla^1_g)^2 a_\varepsilon\|_g\leq\varepsilon^2
  $$
  and the subsets $a^{-1}_\varepsilon(1)\subset M$ exhaust $M$ as $\varepsilon\rightarrow 0$ on complete noncompact manifold $M$,
  where $\nabla^1_g$ is the second canonical connection with respect to the metric $g$ and almost complex structure $J$ on $M$ (cf. \cite{Cha}),
  that is, $\nabla^1_g g=0=\nabla^1_g J$, hence $\nabla^1_g \omega=\nabla^1_g g(J\cdot,\cdot)=0$.
   Here, we only give the case on $\mathbb{R}$.
  Let
  \begin{equation}
  f(x)=\left\{
    \begin{array}{ll}
    \exp(\frac{-1}{x}), & ~ x>0\\
      &   \\
    0, & ~ x\leq 0.
   \end{array}
  \right.
  \end{equation}
  Define
  $$
  \psi(x)=\frac{f(x)}{f(x)+f(1-x)}.
  $$
  Note that

  $\bullet$ $0\leq \psi(x)\leq 1$ for $0<x<1$,

  $\bullet$ if $x\leq 0$, $\psi(x)=0$ and if $x\geq 1$, $\psi(x)=1$,

  $\bullet$ $\psi$, $\psi'$ and $\psi''$ are all bounded.\\
  Finally for $x\geq0$, let
  $$
  a_\varepsilon=\psi^2(2-\varepsilon x).
  $$
  Clearly, $a_\varepsilon(x)=1$ on $[0,\frac{1}{\varepsilon}]$ and $a_\varepsilon(x)=0$ on $[\frac{2}{\varepsilon},\infty)$.
  For $\frac{1}{\varepsilon}<x<\frac{2}{\varepsilon}$, we have
  $$
  a'_\varepsilon(x)=-2\varepsilon\psi(2-\varepsilon x)\psi'(2-\varepsilon x).
  $$
  Since $\psi'$ is bounded, we see that $-\varepsilon C_1\sqrt{a_\varepsilon}\leq a'_\varepsilon(x)\leq 0$ for some constant $C_1$.
  Moreover,
  $$
   a''_\varepsilon(x)=2\varepsilon^2\psi'^2(2-\varepsilon x)+2\varepsilon^2\psi(2-\varepsilon x)\psi''(2-\varepsilon x).
  $$
  Since $\psi$, $\psi'$ and $\psi''$ are bounded, we have
  $|a''_\varepsilon(x)|\leq C_2\varepsilon^2$ for some constant $C_2$.

    \vskip 12pt

   Let $(M,g,J,\omega)$ be a $2n$-dimensional complete, noncompact almost K\"ahler manifold with bounded geometry.
   Then $L^2_l\Omega^k(M)$, $0\leq k\leq 2n$, is completion of $\Omega^k_c(M)$ for any non-negative integer $l$.
   $\Delta_{d+d^{\Lambda}}$ and $D_{d+d^{\Lambda}}$ are two formal self-adjoint fourth-order differential operator,
   that is,
   \begin{eqnarray}
    <\Delta_{d+d^{\Lambda}}u,v>  &=&   <u,\Delta_{d+d^{\Lambda}}v>,  \nonumber\\
      <D_{d+d^{\Lambda}}u,v>  &=&   <u,D_{d+d^{\Lambda}}v> ,
   \end{eqnarray}
   for $u,v\in\Omega^k_c(M)$.
   If  $D_{d+d^{\Lambda}}$ is essentially self-adjoint elliptic operator, then for $u\in L^2\Omega^k(M)$,
    $D_{d+d^{\Lambda}}u=0$ (as a distribution) implies that $u\in L^2_4\Omega^k(M)$ (cf. \cite{Ati,L¨¹ck,Shubin}).
    In fact, $D_{d+d^{\Lambda}}:\Omega^k_c(M)\rightarrow\Omega^k_c(M)$ is an elliptic operator of fourth-order, $0\leq k\leq 2n$,
    that is,
    \begin{equation}
       <D_{d+d^{\Lambda}}u,v> = <u,D_{d+d^{\Lambda}}v>, \,\,\,u,v\in\Omega^k_c(M).
    \end{equation}
    When fourth-order elliptic operator $D_{d+d^{\Lambda}}$ is essentially self-adjoint in $L^2\Omega^k(M)$,
    then its closure is a self-adjoint operator in $L^2\Omega^k(M)$
    with the domain
  $$dom(D_{d+d^{\Lambda},min})=dom(D_{d+d^{\Lambda},max})=L^2_4\Omega^k(M),$$
   hence,
    for any $u\in L^2\Omega^k(M)$,
    $D_{d+d^{\Lambda}}u=0$ (in the sense of distribution) implies that $u\in L^2_4\Omega^k(M)$
   (cf. M. Shubin \cite[Theorem on page 18-6]{Shubin} or W. L\"{u}ck \cite[Lemma 1.75]{L¨¹ck}).
    Recall that
    $$
  \Delta_{d+d^{\Lambda}}=dd^{\Lambda}(dd^{\Lambda})^*+\lambda(d^*d+d^{\Lambda*}d^{\Lambda}),\,\lambda>0
  $$
  and
  \begin{eqnarray*}
     D_{d+d^{\Lambda}}&=&  dd^{\Lambda}(dd^{\Lambda})^*+(dd^{\Lambda})^*(dd^{\Lambda})+d^*d^{\Lambda}d^{\Lambda*}d  \\
     && +d^{\Lambda*}dd^*d^{\Lambda}+\lambda(d^*d+d^{\Lambda*}d^{\Lambda}),\,\lambda>0,
  \end{eqnarray*}
 it is easy to get that $D_{d+d^{\Lambda}}$ and $\Delta_{d+d^{\Lambda}}$ have the same kernel (cf. \cite{Ati,Shubin}).
    For  $D_{dd^{\Lambda}}$ and $\Delta_{dd^{\Lambda}}$, we have the similar results.

  With the Gaffney cutoff trick, we get the following lemma.
    \begin{lem}\label{key lemma}
   Suppose that $(M,g,J,\omega)$ is a $2n$-dimensional complete, noncompact almost K\"ahler manifold with bounded geometry.
   If $D_{d+d^{\Lambda}}$ is of essential self-adjointness, and
   an $L^2$ $k$-form $\alpha$, $0\leq k\leq 2n$, is $d+d^\Lambda$-harmonic, i.e., $\Delta_{d+d^{\Lambda}}\alpha=0$,
   then $\alpha$ satisfies $d\alpha=d^{\Lambda}\alpha=0 \,\,\, and \,\,\, (dd^{\Lambda})^*\alpha=0$.
    \end{lem}
    To prove the above lemma, we need point-wise estimate for $d\alpha$, $\alpha\in \Omega^*(M)$.
    Suppose that $(M,g,J,\omega)$ is a $2n$-dimensional complete, noncompact almost K\"ahler manifold with bounded geometry.
    For $p\in M$, choose a local unitary frame $\{e_1,\cdot\cdot\cdot, e_n\}$ for $T^{1,0}M$ near $p$ with respect to the almost Hermitian inner
    product induced from $g$, and let $\{\theta^1,\cdot\cdot\cdot, \theta^n\}$ be a dual coframe.
    Let $\nabla^1_g$ be the second canonical connection with respect to the metric $g$.
   Locally there exists a matrix of complex valued $1$-forms $\{\theta^i_j\}$, called the connection $1$-forms, such that
    $$
    \nabla^1_ge_i=\theta^j_ie_j,\,\,\,\theta^j_i(p)=0.
    $$
    Hence
    \begin{equation}\label{}
      \nabla^1_ge_i|_p=0,\,\,\,\nabla^1_g\theta^i|_p=0, \,\,\,1\leq i\leq n.
    \end{equation}
    It is easy to see that $\{\theta^j_i\}$ satisfies the skew-Hermitian property $\theta^j_i+\overline{\theta^i_j}=0$.
    Now define the torsion $\Theta=(\Theta^1,\cdot\cdot\cdot,\Theta^n)$ of $\nabla^1_g$ by
    \begin{equation}\label{str equ}
      d\theta^i=-\theta^i_j\wedge \theta^j+\Theta^i,
    \end{equation}
    for $i=1,\cdot\cdot\cdot,n$.
    Notice that the $\Theta^i$ are $2$-forms.
    Equation (\ref{str equ}) is known as the first structure equation.
    Define the curvature $\Psi=\{\Psi^i_j\}$ of $\nabla^1_g$ by
    \begin{equation}\label{str equ1}
     d\theta^i_j=-\theta^i_k\wedge\theta^k_j+\Psi^i_j.
    \end{equation}
    Note that $\{\Psi^i_j\}$ is a skew-Hermitian matrix of $2$-froms.
     Equation (\ref{str equ1}) is known as the second structure equation.
     Since $d\omega=0$,
     \begin{equation}\label{}
       \Theta^i=N^i_{\bar{l}\bar{m}}\overline{\theta^l}\wedge\overline{\theta^m},
     \end{equation}
     where $N^i_{\bar{l}\bar{m}}$ is the Nijenhuis tensor which is defined as
     $$
     N_J(X,Y)=[JX,JY]-J[JX,Y]-J[X,JY]-[X,Y],\,\,\,X,Y\in TM.
     $$

     Now, let us estimate $d\alpha$.
     For simplicity, $\alpha\in \Omega^{1,1}(M)$ is locally written as $\alpha=\alpha_{i\bar{j}}\theta^i\wedge\overline{\theta^j}$.
     \begin{eqnarray*}
       d\alpha &=& d(\alpha_{i\bar{j}}\theta^i\wedge\overline{\theta^j}) \\
        &=& (\nabla^1_g\alpha_{i\bar{j}})\wedge\theta^i\wedge\overline{\theta^j}
        +\alpha_{i\bar{j}}(d\theta^i)\wedge\overline{\theta^j}-\alpha_{i\bar{j}}\theta^i\wedge(d\overline{\theta^j}) \\
        &=& (\nabla^1_g\alpha_{i\bar{j}})\wedge\theta^i\wedge\overline{\theta^j}
        +\alpha_{i\bar{j}}N^i_{\bar{l}\bar{m}}\overline{\theta^j}\wedge\overline{\theta^l}\wedge\overline{\theta^m}
        -\alpha_{i\bar{j}}\overline{N^i_{\bar{l}\bar{m}}}\theta^j\wedge\theta^l\wedge\theta^m.
     \end{eqnarray*}
     Since $(M,g,\omega)$ is of bounded geometry, by the definition of $N^i_{\bar{l}\bar{m}}$ (cf. \cite{MS})
     (notice that $N_J(X,Y)=2J(\nabla_XJ)Y-2J(\nabla_YJ)X$, $X,Y\in TM$, where $\nabla$ is the Levi-Civita connection induced from the metric $g$),
     $N^i_{\bar{l}\bar{m}}$ is uniformly bounded on $M$.
     Hence,
     \begin{equation}\label{estimate}
       \|d\alpha\|_g\leq C(J,\omega)( \|\nabla^1_g\alpha\|_g+\|\alpha\|_g),
     \end{equation}
   where $\|\alpha\|^2_g=(\alpha,\alpha)_g$.

  {\bf Proof of Lemma \ref{key lemma}}
  We want again to justify the integral identity
  $$
  <\Delta_{d+d^{\Lambda}}\alpha,\alpha>_g=<(dd^{\Lambda})^*\alpha,(dd^{\Lambda})^*\alpha>_g+\lambda<d\alpha,d\alpha>_g+\lambda<d^{\Lambda}\alpha,d^{\Lambda}\alpha>_g.
  $$
  If $d\alpha$, $d^*\alpha$, $d^\Lambda\alpha$, $d^{\Lambda*}\alpha$, $dd^\Lambda\alpha$ and $(dd^\Lambda)^*\alpha$ are all $L^2$ forms,
  then the above equation follows by Lemma \ref{our lemma}.

  To handle the general case, we will use the Gaffney cutoff trick.
  Let $\alpha\in\ker\Delta_{d+d^{\Lambda}}=\ker D_{d+d^{\Lambda}}$.
  We cutoff $\alpha$ and obtain by a simple computation
   $0= <\Delta_{d+d^{\Lambda}}\alpha,a_\varepsilon\alpha>_g=I_1(\varepsilon)+I_2(\varepsilon)$,
   where
   $$
   I_1(\varepsilon)=\int_Ma_\varepsilon(\lambda\|d\alpha\|^2_g+\lambda\|d^\Lambda\alpha\|^2_g+\|(dd^\Lambda)^*\alpha\|^2_g)
   $$
   and
   \begin{eqnarray*}
  |I_2(\varepsilon)|&\leq & C_1\int_M\|\nabla^1_g a_\varepsilon\|_g\cdot\|\alpha\|_g\cdot
   (\|(dd^\Lambda)^*\alpha\|_g+\lambda\|d\alpha\|_g+\lambda\|d^\Lambda\alpha\|_g) \\
      && + C_2\int_M\|(dd^\Lambda)^*\alpha\|_g\cdot\|\nabla^1_g a_\varepsilon\|_g\cdot\|\nabla^1_g \alpha\|_g  \\
      && + C_3\int_M\|(dd^\Lambda)^*\alpha\|_g\cdot\|(\nabla^1_g)^2 a_\varepsilon\|_g\cdot\|\alpha\|_g
   \end{eqnarray*}
   where $\|d\alpha\|^2_g=(d\alpha,d\alpha)_g$ and $C_1,C_2,C_3$ are some positive constants.
   Indeed, since on $k$-forms, $d^\Lambda=(-1)^{k+1}*_sd*_s$ and $(dd^\Lambda)^*=(-1)^{k+1}*_gdd^\Lambda*_g$,
   we have
  \begin{eqnarray*}
      <\Delta_{d+d^{\Lambda}}\alpha,a_\varepsilon\alpha>_g &=& <(dd^{\Lambda})^*\alpha,(dd^{\Lambda})^*a_\varepsilon\alpha>_g\\
      &&+\lambda<d\alpha,da_\varepsilon\alpha>_g+\lambda<d^{\Lambda}\alpha,d^{\Lambda}a_\varepsilon\alpha>_g  \\
      &=& <(dd^{\Lambda})^*\alpha,(-1)^{k+1}*_gdd^{\Lambda}*_ga_\varepsilon\alpha>_g\\
      &&+\lambda<d^{\Lambda}\alpha,(-1)^{k+1}*_sd*_sa_\varepsilon\alpha>_g \\
      &&+\lambda<d\alpha,da_\varepsilon\wedge\alpha>_g+\lambda<d\alpha,a_\varepsilon d\alpha>_g   \\
       &=& <(dd^{\Lambda})^*\alpha,*_gd*_sd*_s*_ga_\varepsilon\alpha>_g  \\
       &&+\lambda<d^{\Lambda}\alpha,(-1)^{k+1}*_sda_\varepsilon\wedge*_s\alpha>_g \\
         &&+\lambda<d^{\Lambda}\alpha,(-1)^{k+1}a_\varepsilon*_sd*_s\alpha>_g \\
       &&  +\lambda<d\alpha,da_\varepsilon\wedge\alpha>_g+\lambda<d\alpha,a_\varepsilon d\alpha>_g   \\
       &=& <(dd^{\Lambda})^*\alpha,a_\varepsilon*_gd*_sd*_s*_g\alpha>_g  \\
       &&+<(dd^{\Lambda})^*\alpha,*_gd a_\varepsilon\wedge*_sd*_s*_g\alpha>_g  \\
      &&+ <(dd^{\Lambda})^*\alpha,*_gd*_sda_\varepsilon\wedge*_s*_g\alpha>_g  \\
       &&+\lambda<d^{\Lambda}\alpha,(-1)^{k+1}*_sda_\varepsilon\wedge*_s\alpha>_g \\
         &&+\lambda<d^{\Lambda}\alpha,(-1)^{k+1}a_\varepsilon*_sd*_s\alpha>_g \\
       &&  +\lambda<d\alpha,da_\varepsilon\wedge\alpha>_g+\lambda<d\alpha,a_\varepsilon d\alpha>_g.
   \end{eqnarray*}
   Let
   \begin{eqnarray*}
    I_2(\varepsilon)&=&+<(dd^{\Lambda})^*\alpha,*_gd a_\varepsilon\wedge*_sd*_s*_g\alpha>_g  \\
      &&+ <(dd^{\Lambda})^*\alpha,*_gd*_sda_\varepsilon\wedge*_s*_g\alpha>_g  \\
       &&+\lambda<d^{\Lambda}\alpha,(-1)^{k+1}*_sda_\varepsilon\wedge*_s\alpha>_g \\
       &&  +\lambda<d\alpha,da_\varepsilon\wedge\alpha>_g.
   \end{eqnarray*}
   Then
   \begin{eqnarray*}
    I_1(\varepsilon)&=&<(dd^{\Lambda})^*\alpha,a_\varepsilon(dd^{\Lambda})^*\alpha>_g\\
      &&+\lambda<d\alpha,a_\varepsilon d\alpha>_g+\lambda<d^{\Lambda}\alpha,a_\varepsilon d^{\Lambda}\alpha>_g .
   \end{eqnarray*}
   Since $\nabla^1_gJ=0$, $\nabla^1_gg=0$ and $\nabla^1_g\omega=0$,
   then $\nabla^1_g*_g=0$ and $\nabla^1_g*_s=0$.
   By (\ref{estimate}), we have
   \begin{eqnarray*}
      \|*_gd a_\varepsilon\wedge*_sd*_s*_g\alpha\|_g &\leq& C \{\|a'_\varepsilon\nabla^1_g(*_s*_g)\alpha\|_g+\|a'_\varepsilon\alpha\|_g\}\\
      &\leq& C \{\|a'_\varepsilon\nabla^1_g\alpha\|_g+\|a'_\varepsilon\alpha\|_g\},
   \end{eqnarray*}
   \begin{eqnarray*}
      \|*_gd*_sda_\varepsilon\wedge*_s*_g\alpha\|_g &\leq&
      C \{\|\nabla^1_g(*_sda_\varepsilon\wedge*_s*_g\alpha)\|_g+\|*_sda_\varepsilon\wedge*_s*_g\alpha\|_g\}\\
      &\leq& C \{\|a'_\varepsilon\nabla^1_g\alpha\|_g+\|a''_\varepsilon\alpha\|_g+\|a'_\varepsilon\alpha\|_g\}.
   \end{eqnarray*}
   Therefore,
   \begin{eqnarray*}
     |I_2(\varepsilon)| &\leq&C_2\int_M\|(dd^{\Lambda})^*\alpha\|_g\|\nabla^1_ga_\varepsilon\|_g \|\nabla^1_g\alpha\|_g \\
      &+& C_3 \int_M\|(dd^{\Lambda})^*\alpha\|_g\|(\nabla^1_g)^2a_\varepsilon\|_g \|\alpha\|_g   \\
      &+& C_4 \int_M\|(dd^{\Lambda})^*\alpha\|_g\|\nabla^1_ga_\varepsilon\|_g \|\alpha\|_g \\
      &+& C_5 \int_M (\lambda\|d\alpha\|_g+\lambda\|d^\Lambda\alpha\|_g)\|\nabla^1_ga_\varepsilon\|_g \|\alpha\|_g,
   \end{eqnarray*}
    where $C_4,C_5$ are some positive constants and $C_1=\max\{C_4,C_5\}$.
   Without loss of generality, we assume $C_1=C_2=C_3=1$.
   Choose cutoff functions $a_\varepsilon$, such that
  $$
  a_\varepsilon\geq0,\,\,\,\|\nabla^1_g a_\varepsilon\|_g\leq\varepsilon (a_\varepsilon)^{\frac{1}{2}}, \,\,\, \|(\nabla^1_g)^2 a_\varepsilon\|_g\leq\varepsilon^2
  $$
  and the subsets $a^{-1}_\varepsilon(1)\subset M$ exhaust $M$ as $\varepsilon\rightarrow 0$.
  Then
  \begin{eqnarray*}
    |I_2(\varepsilon)| &\leq& \varepsilon\int_M(a_\varepsilon)^{\frac{1}{2}}\cdot\|\alpha\|_g\cdot
   (\|(dd^\Lambda)^*\alpha\|_g+\lambda\|d\alpha\|_g+\lambda\|d^\Lambda\alpha\|_g) \\
   &&+\varepsilon\int_M(a_\varepsilon)^{\frac{1}{2}}\|(dd^\Lambda)^*\alpha\|_g\cdot\|\nabla^1_g \alpha\|_g
   +\varepsilon^2\int_M\|(dd^\Lambda)^*\alpha\|_g\cdot\|\alpha\|_g \\
     &\leq& \varepsilon\|\alpha\|_{L^2}[(\int_Ma_\varepsilon\|(dd^\Lambda)^*\alpha\|^2_g)^{\frac{1}{2}}
     +(\int_Ma_\varepsilon\lambda\|d\alpha\|^2_g)^{\frac{1}{2}}
     +(\int_Ma_\varepsilon\lambda\|d^\Lambda\alpha\|^2_g)^{\frac{1}{2}}] \\
      &&+\varepsilon\int_M(a_\varepsilon)^{\frac{1}{2}}\|(dd^\Lambda)^*\alpha\|_g\cdot\|\nabla^1_g \alpha\|_g
       +\varepsilon^2\int_M\|(dd^\Lambda)^*\alpha\|_g\cdot\|\alpha\|_g  \\
     &\leq& 2\varepsilon\|\alpha\|_{L^2}(\int_Ma_\varepsilon\|(dd^\Lambda)^*\alpha\|^2_g
     +\int_Ma_\varepsilon\lambda\|d\alpha\|^2_g
     +\int_Ma_\varepsilon\lambda\|d^\Lambda\alpha\|^2_g)^{\frac{1}{2}} \\
      &&+\varepsilon\int_M(a_\varepsilon)^{\frac{1}{2}}\|(dd^\Lambda)^*\alpha\|_g\cdot\|\nabla^1_g \alpha\|_g
       +\varepsilon^2\int_M\|(dd^\Lambda)^*\alpha\|_g\cdot\|\alpha\|_g \\
     &=& 2\varepsilon\|\alpha\|_{L^2}\cdot I_1(\varepsilon)^{\frac{1}{2}}+\varepsilon\int_M(a_\varepsilon)^{\frac{1}{2}}\|(dd^\Lambda)^*\alpha\|_g\cdot\|\nabla^1_g \alpha\|_g\\
     && +\varepsilon^2\int_M\|(dd^\Lambda)^*\alpha\|_g\cdot\|\alpha\|_g \\
     &\leq&4\varepsilon^2\|\alpha\|^2_{L^2}+\frac{1}{4}I_1(\varepsilon)+\varepsilon\int_M(a_\varepsilon)^{\frac{1}{2}}\|(dd^\Lambda)^*\alpha\|_g\cdot\|\nabla^1_g \alpha\|_g\\
     && +\varepsilon^2\int_M\|(dd^\Lambda)^*\alpha\|_g\cdot\|\alpha\|_g
  \end{eqnarray*}
  \begin{eqnarray*}
     &\leq&4\varepsilon^2\|\alpha\|^2_{L^2}+\frac{1}{4}I_1(\varepsilon)+\varepsilon\int_M\|(dd^\Lambda)^*\alpha\|_g\cdot\|\nabla^1_g \alpha\|_g\\
     && +\varepsilon^2\int_M\|(dd^\Lambda)^*\alpha\|_g\cdot\|\alpha\|_g.
  \end{eqnarray*}
  Since $I_1(\varepsilon)=|I_2(\varepsilon)|$,
  note that $(dd^\Lambda)^*\alpha$ and $\nabla^1_g \alpha$ are in $L^2$-space since $D_{d+d^\Lambda}$ is of essential self-adjointness
  (hence, $dom(D_{d+d^{\Lambda},min})=dom(D_{d+d^{\Lambda},max})$, $\ker\Delta_{d+d^{\Lambda}}=\ker D_{d+d^{\Lambda}}$,
 and $D_{d+d^{\Lambda}}\alpha=0$ in the sense of distribution implies that $\alpha\in L^2_4\Omega^k$),
  we get
  $$I_1(\varepsilon)\leq\frac{16}{3}\varepsilon^2\|\alpha\|^2_{L^2}+\frac{4}{3}\varepsilon\int_M\|(dd^\Lambda)^*\alpha\|_g\cdot\|\nabla^1_g \alpha\|_g
  +\frac{4}{3}\varepsilon^2\int_M\|(dd^\Lambda)^*\alpha\|_g\cdot\|\alpha\|_g ,$$
  and hence $I_1(\varepsilon)\rightarrow 0$ as $\varepsilon\rightarrow 0$.
  This completes the proof of Lemma \ref{key lemma}.  ~~~$\Box$

    \vskip 12pt

   Let $(M,g,J,\omega)$ be a $2n$-dimensional complete noncompact almost K\"ahler manifold with bounded geometry.
   It is same to the classical de Rham Laplacian operator $\Delta_d$,
   we can define $L^2$-symplectic cohomology groups as follows:
   \begin{equation}
     L^2H^k_{d+d^\Lambda}(M)
     =\frac{\ker(d+d^\Lambda)\cap  L^2\Omega^k}{\text{im}\, dd^\Lambda\cap  L^2\Omega^k},
   \end{equation}
   \begin{equation}
     L^2H^k_{dd^\Lambda}(M)
     =\frac{\ker(dd^\Lambda)\cap  L^2\Omega^k}{(\text{im}\, d+\text{im}\,d^\Lambda)\cap  L^2\Omega^k}.
   \end{equation}
   Since $D_{d+d^\Lambda}$ is fourth-order elliptic operator,
   if $D_{d+d^\Lambda}$ is of essential self-adjointness,
   then $\Delta_{d+d^\Lambda}$ and $D_{d+d^\Lambda}$  have the same kernel.
   For $\Delta_{dd^\Lambda}$ and $D_{dd^\Lambda}$, we have the similar results.

  With Lemma \ref{our lemma} and Lemma \ref{key lemma} one concludes, as in the closed manifolds,
   that the $L^2\Omega^k(M)$ of exterior $k$-forms on a complete manifold $M$ with bounded geometry admits the Hodge decomposition (cf. \cite{Ati,L¨¹ck,Shubin}).
    \begin{prop}\label{Hodge deco1}
     Suppose $(M,g,J,\omega)$ is a complete, noncompact almost K\"ahler manifold with bounded geometry,
     and $D_{d+d^{\Lambda}}$, $D_{dd^{\Lambda}}$ are of essential self-adjointness.
     Then
     $$
     L^2\Omega^k=L^2\mathcal{H}^k_{d+d^\Lambda}\oplus \overline{dd^\Lambda (L^2\Omega^k)}\oplus[\overline{d^*(L^2\Omega^{k+1})}+\overline{d^{\Lambda*}(L^2\Omega^{k-1})}],
    $$
   where $\overline{d(\cdot\cdot\cdot)}$ is the closure in $L^2\Omega^k$ of the intersection of $L^2\Omega^k$ with the
  image of $d$.
  Similarly, we can get
    $$
     L^2\Omega^k=L^2\mathcal{H}^k_{dd^\Lambda}\oplus \overline{(dd^\Lambda)^*
(L^2\Omega^k)}\oplus[\overline{d(L^2\Omega^{k-1})}+\overline{d^{\Lambda}(L^2\Omega^{k+1})}].
     $$
   \end{prop}

   If $D_{d+d^{\Lambda}}$ is an essentially self-adjoint elliptic operator,
    then $\Delta_{d+d^\Lambda}$ and $D_{d+d^\Lambda}$  have the same kernel.
    It is easy to get an isometric isomorphism
    $$
    L^2\mathcal{H}^k_{d+d^\Lambda} \cong L^2H^k_{d+d^\Lambda}.
    $$
   Similarly, if $D_{dd^{\Lambda}}$ is an essentially self-adjoint elliptic operator,
   then
      $$
    L^2\mathcal{H}^k_{dd^\Lambda} \cong L^2H^k_{dd^\Lambda}.
    $$
       More details, see \cite{Ati,L¨¹ck,Shubin}.

Similarly to Corollary $3.8$ and Corollary $3.19$ in \cite{TY},
   we get the following isomorphism.

\begin{prop}\label{Hodge decoL2}
     Suppose that $(M,g,J,\omega)$ is a $2n$-dimensional complete,
     noncompact almost K\"ahler manifold with bounded geometry.
     If $D_{d+d^{\Lambda}}$ and $D_{dd^{\Lambda}}$ are of essential self-adjointness,
     then the Lefschetz operator defines isomorphisms
     $$
   L^{n-k}:L^2\mathcal{H}^k_{d+d^\Lambda}\cong L^2\mathcal{H}^{2n-k}_{d+d^\Lambda},
   $$
    $$
   L^{n-k}:L^2\mathcal{H}^k_{dd^\Lambda}\cong L^2\mathcal{H}^{2n-k}_{dd^\Lambda}.
   $$
   \end{prop}
 \begin{proof}
  For an almost K\"ahler manifold $(M,g,J,\omega)$,
  there exists a connection called second canonical connection $\nabla^1_g$ whose torsion tensor has vanishing $(1,1)$-part.
  Moreover,
  $\nabla^1_g$ satisfies $\nabla^1_gg=0$, $\nabla^1_gJ=0$ and $\nabla^1_g\omega=0$ (cf. \cite{Cha,MS}).
  Since $\nabla^1_g\omega=0$, it implies that $\omega$ is bounded.
    We only prove the first isomorphism, and the other is similar.
    Since $\omega$ is bounded, if $\alpha$ is $L^2$, the same $L^{n-k}\alpha$.
   Note that $\Delta_{d+d^\Lambda}$ preserves the degree of forms
   and $[\Delta_{d+d^\Lambda},L]=0$, $[\Delta_{d+d^\Lambda},\Lambda]=0$ (\cite[Lemma 3.7]{TY}),
  by Proposition \ref{Hodge deco1},
   we get that $L^{n-k}:L^2\mathcal{H}^k_{d+d^\Lambda}\rightarrow L^2\mathcal{H}^{2n-k}_{d+d^\Lambda}$ is an isomorphism.
   \end{proof}

   Decomposition (\ref{Hodge deco}) and Proposition \ref{Hodge decoL2} lead the Lefschetz vanishing property which is similar with Hitchin's result
   (see \cite[Theorem 2]{Hitchin}).

   \vskip 6pt

  {\bf Proof of Theorem \ref{main 1}.}
  It is clear that the symplectic form $\omega$ is bounded with respect to the given almost K\"ahler metric $g$.
  By hypothesis, we assume that $\omega=d\eta$, where $\eta$ satisfy
  $$
  \|\eta(x)\|_g\leq c(\rho(x_0,x)+1).
  $$
  Then for every closed $L^2$ $k$-form $\alpha$, $k<n$, the form $L^{n-k}\alpha=\omega^{n-k}\wedge\alpha=d\beta$,
  where $\beta=\eta\wedge\omega^{n-k-1}\wedge\alpha$.
  By \cite[Proof of Theorem 1]{Hitchin}, we obtain that $\omega^{n-k}=\eta\wedge\omega^{n-k-1}$ is $d$(sublinear). Then applying \cite[Theorem 1]{Hitchin} again we get that $L^{n-k}\alpha=d(\eta\wedge\omega^{n-k-1}\wedge\alpha)=d\beta$ lies in the closure of $dL^2\Omega^{2n-k-1}\cap L^2\Omega^{2n-k}$.
  In particular, if $\alpha$ is $dd^\Lambda$-harmonic,
   by Proposition \ref{Hodge decoL2}, $L^{n-k}\alpha$ is also $dd^\Lambda$-harmonic,
  i.e., $d\beta$ is $dd^\Lambda$-harmonic.
  Then $d\beta$ is  $\Delta_d$-harmonic.
  Hence, by decomposition (\ref{Hodge deco}),
  $L^{n-k}\alpha=d\beta=0$.
  Proposition \ref{Hodge decoL2} has stated that $L^{n-k}$ is an isomorphism
   from $L^2\mathcal{H}^k_{dd^\Lambda}$ to $L^2\mathcal{H}^{2n-k}_{dd^\Lambda}$ for $k<n$.
   Therefore, $\alpha=0$.
   At last, we can summarize that both $L^2\mathcal{H}^k_{dd^\Lambda}=0$ for $k<n$ and $L^2\mathcal{H}^k_{dd^\Lambda}=0$ for $k>n$.

   By simple calculation, we find that the Laplacians $\Delta_{d+d^\Lambda}$ and $\Delta_{dd^\Lambda}$ satisfy
   $$
   *_g\Delta_{d+d^\Lambda}=\Delta_{dd^\Lambda}*_g.
   $$
   Hence, we have
   $
   *_g: L^2\mathcal{H}^k_{dd^\Lambda}\rightarrow L^2\mathcal{H}^{2n-k}_{d+d^\Lambda}
   $
   is an isomorphism.
   Therefore, we can also summarize that $L^2\mathcal{H}^k_{d+d^\Lambda}=0$, unless $k=n$.\qed

 \section{Symplectic Euler characteristics  }\setcounter{equation}{0}

    A Hilbert space $\mathcal{H}$ with a unitary action of a countable group $\Gamma$ is called a $\Gamma$-module if
   $\mathcal{H}$ is isomorphic to a $\Gamma$-invariant subspace in the space of $L^2$-functions on $\Gamma$ with values in some Hilbert space $H$.
  To each $\Gamma$-module, one assigns the Von Neumann dimension, also called $\Gamma$-dimension,
   $0\leq \dim_{\Gamma}\mathcal{H}\leq \infty$, which is a nonnegative real number or $+\infty$ (see \cite{Ati,L¨¹ck,Pan,Shubin}).
  The precise definition is not important for the moment but the following properties convey the idea of $\dim_{\Gamma}\mathcal{H}$ as some kind of size of the
   ``quotient space" $\mathcal{H}/\Gamma$:

     (i) $\dim_{\Gamma}\mathcal{H}=0$ $\Leftrightarrow$ $\mathcal{H}=0$.

     (ii) If $\Gamma$ is a finite group, then $\dim_{\Gamma}\mathcal{H}=\dim\mathcal{H}/card\Gamma$.

     (iii)$\dim_{\Gamma}\mathcal{H}$ is additive. Given $0\rightarrow\mathcal{H}_1\rightarrow\mathcal{H}_2\rightarrow\mathcal{H}_3\rightarrow0$,
     one has $\dim_{\Gamma}\mathcal{H}_2=\dim_{\Gamma}\mathcal{H}_1+\dim_{\Gamma}\mathcal{H}_3$.

     (iv) If $\mathcal{H}$ equals the space of $L^2$-functions $\Gamma\rightarrow H$,
    then $\dim_{\Gamma}\mathcal{H}=\dim H$. In particular, if $H=\mathbb{R}^n$, then $\dim_{\Gamma}\mathcal{H}=n$.

 Here we are interested in the situation where $\Gamma$ is a discrete faithful group of symplectomorphisms of a symplectic manifold $(M,\omega)$.
Find an almost K\"ahler structure $(g,J,\omega)$ on $(M,\omega)$, then one has an almost K\"ahler manifold (cf \cite{MS}). It is not hard to see that the given group $\Gamma$ acts on $(M,g,J,\omega)$ as deck transformation group \cite{Cha}. One can easily show that the spaces $L^2\mathcal{H}^k_{d+d^\Lambda}$, $L^2\mathcal{H}^k_{dd^\Lambda}$
   of harmonic $L^2$-forms are $\Gamma$-module for all degrees $k$ (see \cite{Ati,L¨¹ck,Pan,Shubin}),
  and then one defines the $L^2$-symplectic Betti numbers $L^2\beta^{s,1}_k\triangleq\dim_{\Gamma}L^2\mathcal{H}^k_{d+d^\Lambda}$
 and $L^2\beta^{s,2}_k\triangleq\dim_{\Gamma}L^2\mathcal{H}^k_{dd^\Lambda}$.
    The most interesting case is when $M/\Gamma$ is closed. Then the $L^2$-symplectic Betti numbers are finite $L^2\beta^s_k<\infty$
   and the $L^2$-symplectic Euler characteristics is defined by
   $$
  L^2\chi^{s,1}(M)\triangleq\sum^{2n}_{k=0}(-1)^kL^2\beta^{s,1}_k
   $$
   and
   $$
   L^2\chi^{s,2}(M)\triangleq\sum^{2n}_{k=0}(-1)^kL^2\beta^{s,2}_k.
   $$

   First recall how Hodge theory works on a complete non-compact Riemannian manifold (\cite{Hitchin}).
   If $L^2\Omega^k$ denotes the Hilbert space of $L^2$ $k$-forms, then the $L^2$-de Rham cohomology group
   $L^2H^k_{dR}$ is defined as the quotient of the space of closed $L^2$ $k$-forms by the closure of the space
   $$
   dL^2\Omega^{k-1}\cap L^2\Omega^k.
   $$
    Similarly, we can define the $L^2$-symplectic cohomology group on a complete, noncompact almost K\"ahler manifold $(M,g,J,\omega)$
    with bounded geometry by
    $$
     L^2H^k_{d+d^{\Lambda}}=\frac{\ker(d+d^{\Lambda})\cap L^2\Omega^k}{\overline{dd^{\Lambda}L^2\Omega^{k-1}\cap L^2\Omega^k}}.
    $$
 If $D_{d+d^{\Lambda}}$ is of essential self-adjointness
  by decomposition (\ref{Hodge deco}) and Proposition \ref{Hodge deco1}, we can find that
  $$
  L^2H^k_{dR}\cong L^2\mathcal{H}^k_d,\,\,\, L^2H^k_{d+d^{\Lambda}}\cong L^2\mathcal{H}^k_{d+d^\Lambda},
  $$
  and every $L^2$ cohomology class has a $L^2$ harmonic representative form which is the only one.

  \vskip 12pt

   A closed symplectic manifold $(M,\omega)$ is said to satisfy the $dd^\Lambda$-Lemma if every $d$-exact, $d^\Lambda$-closed form is $dd^\Lambda$-exact.
   In fact, it turns out that the following conditions are equivalent on a closed symplectic manifold $(M,\omega)$ (\cite{Angella,Cava,Gu,Mathieu,Merk,TY,YanH}):

   $\bullet$  $(M,\omega)$ satisfies the $dd^{\Lambda}$-Lemma;

   $\bullet$ the natural homomorphism $H^{\bullet}_{d+d^\Lambda}(M;\mathbb{R})\rightarrow H^{\bullet}_{dR}(M;\mathbb{R})$ is actually an isomorphism;

   $\bullet$ every de Rham cohomology class admits a representative being both $d$-closed and $d^\Lambda$-closed;

   $\bullet$ the hard Lefschetz Condition holds on $(M,\omega)$.

   More generally, Dong Yan has gotten the following result on a symplectic manifold which may be not compact.
    \begin{prop}\label{Yan theo}
    ({\rm \cite{YanH}})
    Let $(M,\omega)$  be a symplectic manifold with dimension $2n$.
    Then the following two assertions are equivalent:

    1. every de Rham cohomology class admits a representative being both $d$-closed and $d^\Lambda$-closed;

    2. For any $k\leq n$, the cup product $L^k:H^{n-k}_{dR}(M;\mathbb{R})\rightarrow H^{n+k}_{dR}(M;\mathbb{R})$ is surjective.
    \end{prop}

    \begin{rem}
    The above assertion 2 is just the definition of hard Lefschetz property on a symplectic manifold which may be not compact.
    \end{rem}

    Let $(M,g,J,\omega)$ be a complete noncompact almost K\"ahler manifold with bounded geometry.
   Denote the Sobolev space
    $$\displaystyle{L^2_l\Omega^{k}=\{ \alpha\in \Omega^{k}~|~\sum_{i=0}^l|(\nabla^1_g)^i\alpha|_g\in L^2(M)\}},$$
    where $\nabla^1_g$ be the second canonical connection with respect to the given almost K\"ahler structure $(g,J,\omega)$ (cf. \cite{Cha,MS}).
     Now we define the $L^2$-$dd^\Lambda$-Lemma on a complete noncompact almost K\"ahler manifold.
   \begin{defi}
    Let $(M,g,J,\omega)$ be a complete noncompact almost K\"ahler manifold with bounded geometry.
    Let $\alpha\in L^2\Omega^k$ be a $d$- and $d^\Lambda$-closed differential form.
     We say that the $L^2$-$dd^\Lambda$-Lemma holds if the following properties are equivalent:

     (i) $\alpha=d\beta$, $\beta\in L^2_1\Omega^{k-1}$;

     (ii) $\alpha=d^\Lambda\gamma$, $\gamma\in L^2_1\Omega^{k+1}$;

     (ii) $\alpha=dd^\Lambda\theta$, $\theta\in L^2_2\Omega^k$.
     \end{defi}

  \begin{prop}\label{key prop}
    Let $(M,g,J,\omega)$ be a $2n$-dimensional closed almost K\"ahler manifold, $\Pi:(\tilde{M},\tilde{g},\tilde{J},\tilde{\omega})\rightarrow (M,g,J,\omega)$ the universal covering map.
    If $(M,g,J,\omega)$ satisfies the hard Lefschetz property, then $L^2$-$dd^\Lambda$-Lemma holds on $(\tilde{M},\tilde{g},\tilde{J},\tilde{\omega})$.
  \end{prop}
   \begin{proof}
  We may assume without loss of generality that $M$ is a connected manifold.
   Denote by $\pi_1(M)$ the fundamental group of $M$.
   Let $\Gamma$ be the deck transformation group of the covering.
   Then $\Gamma$ is isomorphic to $\pi_1(M)$.
    Notice that $(\tilde{M},\tilde{g},\tilde{J},\tilde{\omega})$ is a complete, noncompact almost K\"ahler manifold with bounded geometry.
   Since $\tilde{D}_{d+d^\Lambda}=\pi^*D_{d+d^\Lambda}$, $\tilde{D}_{dd^\Lambda}=\pi^*D_{dd^\Lambda}$ which are $\Gamma\cong\pi_1(M)$-invariant
   fourth-order elliptic operators,
   $\tilde{D}_{d+d^\Lambda}$ and $\tilde{D}_{dd^\Lambda}$ are essential self-adjointness (cf.\cite{Ati,L¨¹ck,Shubin}).
   Suppose $F\subseteq\tilde{M}$ is the fundamental domain of the universal covering.
   It is well known that $\Pi(F)$ is an open set of $M$ and $\Pi(\bar{F})=M$ (cf. \cite{Cha}),
    moreover both $\partial\bar{F}$ and $M\setminus\Pi(F)$ satisfy the Hausdorff dimension less than or equal to $2n-1$ (cf. \cite{Cha}).
    For any $\phi\in\Gamma$, $\phi:\tilde{M}\rightarrow\tilde{M}$ is a homeomorphism.
    Denote by $F_{\phi}=\phi(F)$, then $\phi:F\rightarrow F_{\phi}$ is a diffeomorphism
    and $F\cap F_{\phi}=\varnothing$ for any $\phi\neq e$.
      Moreover, $\tilde{M}=\cup_{\phi\in\Gamma}\phi(\bar{F})$.

   \vskip 6pt

   Since $(M,\omega)$ satisfies the hard Lefschetz property, then the $dd^\Lambda$-Lemma holds on $M$,
   that is
   $$
   Imd\cap\ker d^\Lambda= \ker d\cap Imd^\Lambda=Imdd^\Lambda.
   $$
   Suppose that $\alpha$ is a $d$-closed $k$-form on $M$ and $\alpha=d^\Lambda\beta$.
   Then there exists $k-1$-form $\gamma$ such that $\alpha=d^\Lambda\beta=d\gamma$.
   Note that $\gamma\in\Omega^{k-1}=\mathcal{H}^{k-1}_d\oplus
     d(\Omega^{k-2})\oplus d^*(\Omega^k)$,
    without loss of generality, we can assume $\gamma=d^*\eta$, where $\eta$ is a $k$-form.
    Then $\alpha=dd^*\eta$. Using the Hodge decomposition again, we can assume $\eta=d\xi$, where $\xi\in\Omega^{k-1}$.
    Since $dd^*:d\Omega^{k-1}\rightarrow d\Omega^{k-1}$ is an elliptic linear operator and essentially self-adjoint (cf.\cite{Ati,L¨¹ck,Shubin}),
    we can obtain
    $$
    \|\eta\|_{L^2_2(M)}\leq c_M\|\alpha\|_{L^2(M)},
    $$
    where $c_M$ is constant which only depends on $M$.
    Indeed, we have gotten the following property in distribution sense.
    If $\alpha=d^\Lambda\beta$ is a $d$-closed $L^2$ form on $M$, then we can find a $L^2_1$ form $\gamma$ such that $\alpha=d\gamma$.
    Since $\Pi: F\hookrightarrow M$ is a diffeomorphism and $M\setminus\Pi(F)$ satisfy the Hausdorff dimension less than or equal to $2n-1$ (cf. \cite{Cha}),
    we obtain that:
     {\it If $\alpha_F=d^\Lambda\beta_F$ is a $d$-closed $L^2$ form on $F$, then we can find a $L^2_1$ form $\gamma_F$ such that $\alpha_F=d\gamma_F$.}
    Suppose that $\tilde{\alpha}=d^\Lambda\tilde{\beta}$ is a $d$-closed $k$-form on $\tilde{M}$, moreover $\tilde{\alpha}$ is $L^2$.
    Restricted to $F_{\phi}$,
    we can find $\tilde{\eta}_{F_{\phi}}$ such that
     $
    \|\tilde{\eta}_{F_{\phi}}\|_{L^2_2(F_{\phi})}\leq c(F_{\phi})\|\tilde{\alpha}\|_{L^2(F_{\phi})}
    $
    and
    $\tilde{\alpha}|_{F_{\phi}}=d\tilde{\gamma}_{F_{\phi}}=dd^*\tilde{\eta}_{F_{\phi}}$,
    where $\tilde{\gamma}_{F_{\phi}}\triangleq d^*\tilde{\eta}_{F_{\phi}}$.
    Define $\tilde{\eta}=\sum_{\phi\in\Gamma}\tilde{\eta}_{F_{\phi}}$
    and $\tilde{\gamma}=d^*\tilde{\eta}$.
     It is easy to see that $\tilde{\eta}\in L^2_2(\cup_{\phi\in\Gamma}F_{\phi})=L^2_2(\cup_{\phi\in\Gamma}\bar{F_{\phi}})=L^2_2(\tilde{M})$.
    Then we will get $\tilde{\alpha}=d\tilde{\gamma}$ and $\tilde{\gamma}\in L^2_1(\tilde{M})$.

   \vskip 6pt

    Suppose that $\alpha$ is a $d^\Lambda$-closed $k$-form on $M$ and $\alpha=d\beta$.
   Then by $dd^\Lambda$-Lemma there exists $k$-form $\gamma$ such that $\alpha=dd^\Lambda\gamma$.
   Note that $\Delta_{d^\Lambda}=d^\Lambda d^{\Lambda*}+d^{\Lambda*}d^\Lambda$ is an ellipitic differential operator (see \cite[Proposition 3.3]{TY}).
   Applying elliptic theory to the $\Delta_{d^\Lambda}$ then implies the Hodge decomposition
   \begin{equation}\label{d}
     \Omega^k(M)=\mathcal{H}^k_{d^\Lambda}(M)\oplus d^\Lambda\Omega^{k+1}(M)\oplus d^{\Lambda*}\Omega^{k-1}(M).
   \end{equation}
   Without loss of generality, we can assume $\gamma=d^{\Lambda*}\eta$ and $\alpha=dd^\Lambda d^{\Lambda*}\eta$.
   Using the Hodge decomposition (\ref{d}) again, we can assume $\eta=d^\Lambda\xi$.
    Since $d^\Lambda d^{\Lambda*}:d^\Lambda\Omega^k\rightarrow d^\Lambda\Omega^k$ is an elliptic linear operator,
      we can obtain
    $$
    \|\eta\|_{L^2_3(M)}\leq c_{M,1}\|d^\Lambda d^{\Lambda*}\eta\|_{L^2_1(M)},
    $$
    where $c_{M,1}$ is constant which only depends on $M$.
    Note that $\alpha=d(d^\Lambda d^{\Lambda*}\eta)$,
    we can assume $d^\Lambda d^{\Lambda*}\eta=d^*\theta$,
   since $d^\Lambda d^{\Lambda*}\eta\in\Omega^{k-1}=\mathcal{H}^{k-1}_d\oplus
     d(\Omega^{k-2})\oplus d^*(\Omega^k)$,
   where $\theta$ is a $k$-form.
   It is well known that $d+d^*$ is an elliptic linear operator (cf. \cite{Ati,Cha}),
   so $d: d^*\Omega^k\rightarrow \Omega^k$ is an elliptic linear operator.
   Hence, $\|d^\Lambda d^{\Lambda*}\eta\|_{L^2_1(M)}\leq c_{M,2}\|\alpha\|_{L^2(M)}$.
   Therefore, we can obtain
    $$
    \|\eta\|_{L^2_3(M)}\leq c_M\|\alpha\|_{L^2(M)},
    $$
    where $c_M$ is constant which only depends on $M$.
     We have gotten the following property in distribution sense.
    If $\alpha=d\beta$ is a $d^\Lambda$-closed $L^2$ form on $M$, then we can find a $L^2_2$ form $\gamma$ such that $\alpha=dd^\Lambda\gamma$.
    At last,
    we obtain that:
     {\it If $\alpha_F=d\beta_F$ is a $d^\Lambda$-closed $L^2$ form on $F$, then we can find a $L^2_2$ form $\gamma_F$ such that $\alpha_F=dd^\Lambda\gamma_F$.}
     Suppose that $\tilde{\alpha}=d\tilde{\beta}$ is a $d^\Lambda$-closed $k$-form on $\tilde{M}$, moreover $\tilde{\alpha}$ is $L^2$.
    Restricted to $F_{\phi}$,
    we can find $\tilde{\eta}_{F_{\phi}}$ such that
     $
    \|\tilde{\eta}_{F_{\phi}}\|_{L^2_3(F_{\phi})}\leq c(F_{\phi})\|\tilde{\alpha}\|_{L^2(F_{\phi})}
    $
    and
    $\tilde{\alpha}|_{F_{\phi}}=dd^\Lambda\tilde{\gamma}_{F_{\phi}}=dd^\Lambda d^{\Lambda*}\tilde{\eta}_{F_{\phi}}$,
    where $\tilde{\gamma}_{F_{\phi}}\triangleq d^{\Lambda*}\tilde{\eta}_{F_{\phi}}$.
    Define $\tilde{\eta}=\sum_{\phi\in\Gamma}\tilde{\eta}_{F_{\phi}}$
    and $\tilde{\gamma}=d^{\Lambda*}\tilde{\eta}$.
     It is easy to see that $\tilde{\eta}\in L^2_3(\cup_{\phi\in\Gamma}F_{\phi})=L^2_3(\cup_{\phi\in\Gamma}\bar{F_{\phi}})=L^2_3(\tilde{M})$,
     since $\tilde{M}\setminus\cup_{\phi\in\Gamma}F_{\phi}$ has Hausdorff dimension$\leq2n-1$.
    Then we will get $\tilde{\alpha}=dd^\Lambda\tilde{\gamma}$ and $\tilde{\gamma}\in L^2_2(\tilde{M})$.
   \end{proof}

   \begin{rem}
   Indeed, in the above Proposition, we have proven that

    $\alpha=d^\Lambda\beta$, $\alpha\in L^2$ and $\alpha$ is $d$-closed $\Rightarrow$  $\alpha=d\gamma$, $\gamma\in L^2_1$;

    $\alpha=d\beta$, $\alpha\in L^2$ and $\alpha$ is $d^\Lambda$-closed $\Rightarrow$   $\alpha=dd^\Lambda\gamma$, $\gamma\in L^2_2$.

   \end{rem}

   \begin{prop}\label{key prop2}
    Let $(M,g,J,\omega)$ be a $2n$-dimensional closed almost K\"ahler manifold,
    $\Pi:(\tilde{M},\tilde{g},\tilde{J},\tilde{\omega})\rightarrow (M,g,J,\omega)$ the universal covering map.
    If $(M,g,J,\omega)$ satisfies the hard Lefschetz property,
   then the canonical homomorphism
   $$L^2H^k_{d+d^{\Lambda}}(\tilde{M};\mathbb{R})\rightarrow L^2H^k_{dR}(\tilde{M};\mathbb{R})$$
   is an isomorphism for all $k$.
   \end{prop}
  \begin{proof}
  Notice that since $D_{d+d^{\Lambda}}$ is a $\Gamma=\pi_1(M)$-invariant elliptic operator on $\tilde{M}$,
  then $D_{d+d^{\Lambda}}$ is of essential self-adjointness (cf. \cite[Lecture 18]{Shubin}).
  Hence, for any $\alpha\in L^2\Omega^k$, $D_{d+d^{\Lambda}}\alpha=0$ in the sense of distribution implies that $\alpha\in L^2_4\Omega^k$.
   Suppose that $\alpha$ is a $d$-closed and $d^{\Lambda}$-closed $L^2$ $k$-form such that $[\alpha]_{dR}=0$, i.e. $\alpha=d\beta$ for some $\beta\in L^2\Omega^{k-1}(\tilde{M})$.
  By the proof of Proposition \ref{key prop},
   we can find $\gamma\in L^2_2(\tilde{M})$ such that
 $\alpha=dd^{\Lambda}\gamma$.
  It follows that $[\alpha]_{d+d^{\Lambda}}=0$. This proves that the homomorphism is injective.

  For any $[\alpha]_{dR}\in L^2H^k_{dR}(\tilde{M};\mathbb{R})$,
  by the decomposition \ref{Hodge deco},
  we can assume $\alpha\in L^2\mathcal{H}^k_d$ without loss of generality.
 If $d^\Lambda\alpha= 0$,
 then $[\alpha]_{d+d^{\Lambda}}\in L^2H^k_{d+d^{\Lambda}}(\tilde{M})$ whose image under this map is $[\alpha]_{dR}$.
 Suppose that $d^\Lambda\alpha\neq0$.
 Note that $\Delta_d\alpha=0$ and $\alpha\in L^2$, therefore we can obtain that $\alpha$ is smooth and
 $\|\alpha\|_{L^2_k(\tilde{M})}\leq c(k)$ for any $k=0,1,2,\cdot\cdot\cdot$.
  Since $dd^\Lambda\alpha=0$, by the proof of Proposition \ref{key prop},
   we can find $\gamma\in L^2_2(\tilde{M})$ such that
   $d^\Lambda\alpha=dd^\Lambda\gamma$.
   Hence, $d^\Lambda(\alpha+d\gamma)=0$ and $d(\alpha+d\gamma)=0$.
   It follows that $[\alpha+d\gamma]_{d+d^{\Lambda}}\in L^2H^k_{d+d^{\Lambda}}(\tilde{M})$ whose image under this map is $[\alpha]_{dR}$.
   This proves that the homomorphism is also surjective. So it is an isomorphism.
   \end{proof}

   Let $(M,\omega)$ be a closed symplectic manifold of dimension $2n$.
   Then one can find an almost K\"ahler structure $(g,J,\omega)$ on $(M,\omega)$. $(M,g,J,\omega)$ is a closed almost K\"ahler manifold of dimension $2n$ (cf. \cite{MS}).
   Let $\Pi: (\tilde{M},\tilde{g},\tilde{J},\tilde{\omega})\rightarrow (M,g,J,\omega)$ be the universal covering map. If $(M,g,J,\omega)$ satisfies the hard Lefschetz property, then from Proposition \ref{key prop2}, by M. Atiyah's result (\cite{Ati}), it is easy to get Theorem \ref{main 3}.

  \vskip 6pt

  {\bf Proof of Theorem \ref{main 2}.}  Let $(M,g,J,\omega)$ be a $2n$-dimensional closed almost K\"ahler parabolic
   manifold which satisfies the hard Lefschetz property,
   $$\Pi:(\tilde{M},\tilde{g},\tilde{J},\tilde{\omega})\rightarrow (M,g,J,\omega)$$ the universal covering map.
    Therefore, by Proposition \ref{key prop2},
    $$L^2H^k_{d+d^{\Lambda}}(\tilde{M};\mathbb{R})\rightarrow L^2H^k_{dR}(\tilde{M};\mathbb{R})$$
   is an isomorphism for all $k$.
   Since $L^2H^k_{d+d^{\Lambda}}(\tilde{M};\mathbb{R})\cong L^2\mathcal{H}^k_{d+d^\Lambda}(\tilde{M};\mathbb{R})$,
   by Theorem \ref{main 1}, we know $L^2H^k_{d+d^{\Lambda}}(\tilde{M};\mathbb{R})=0$ for $k\neq n$.
    Hence, $L^2H^k_{dR}(\tilde{M};\mathbb{R})=0$ for $k\neq n$.
    The Atiyah index theorem for covers \cite{Ati} then gives $(-1)^n\chi(M)\geq0$.
    Then, the conclusion follows. \qed

  \vskip 6pt

  \noindent{\bf Acknowledgements.}\,
  The first author would like to thank Professor Xiaojun Huang for his support.
   The second author would like to thank East China Normal University and Professor Qing Zhou
   for hosting his visit in the spring semester in 2014.
   The second author also would like to thank Wuhan University and Professor Hua Chen for hosting his visit in the autumn semester in 2015.
   The authors dedicate this paper to the memory of Professor Weiyue Ding in deep appreciation for his long-term support of their work.
   At last, The authors are grateful to the referees for their valuable comments and suggestions.

  \vskip 24pt

  \noindent Qiang Tan\\
 Faculty of Science, Jiangsu University, Zhenjiang, Jiangsu 212013, China\\
 e-mail: tanqiang@ujs.edu.cn\\

  \vskip 6pt

  \noindent Hongyu Wang\\
  School of Mathematical Sciences, Yangzhou University, Yangzhou, Jiangsu 225002, China\\
  e-mail: hywang@yzu.edu.cn\\

   \vskip 6pt

  \noindent Jiuru Zhou\\
  School of Mathematical Sciences, Yangzhou University, Yangzhou, Jiangsu 225002, China\\
  e-mail: zhoujr1982@hotmail.com

 \end{document}